\documentclass[conference]{IEEEtran}
\ifCLASSINFOpdf
\else
\fi

\usepackage{mathrsfs}
\usepackage{graphicx}
\usepackage{eurosym}
\usepackage{amssymb}
\usepackage{amsmath}
\usepackage{amsfonts}
\usepackage{epstopdf}
\usepackage{epsf,subfigure}
\usepackage{psfrag}
\usepackage{graphics}
\usepackage{color} 
\usepackage{cite}
\usepackage{array,multirow,pbox}
\usepackage{enumitem}

\newcommand{\rem}[1]{}
\newtheorem{proposition}{Proposition}

\newtheorem{theorem}{Theorem}

\newtheorem{remark}{Remark}
\newtheorem{assumption}{Assumption}

\newenvironment{proof}[1][Proof]{\begin{trivlist}
\item[\hskip \labelsep {\bfseries #1}]}{\end{trivlist}}

\newcommand{\qed}{\nobreak \ifvmode \relax \else
      \ifdim\lastskip<1.5em \hskip-\lastskip
      \hskip1.5em plus0em minus0.5em \fi \nobreak
      \vrule height0.75em width0.5em depth0.25em\fi}

\graphicspath{{Figures/}}

\newcommand{\mysubeq}[2]{\begin{subequations}\label{#1}
\begin{align}
#2
\end{align}\end{subequations}}

\begin{document}
%

\title{\vspace{0.25in}{Identifying Parameter Space for Robust Stability in Nonlinear Networks: A Microgrid Application}}



%
\author{\IEEEauthorblockN{Soumya Kundu,
Wei Du, Sai Pushpak Nandanoori,
Frank Tuffner, and
Kevin Schneider}
\IEEEauthorblockA{Electricity Infrastructure and Buildings Division\\
Pacific Northwest National Laboratory, Richland, WA 99354 USA\\
Email:\,\{soumya.kundu,\,wei.du,\,saipushpak.n,\,frank.tuffner,\,kevin.schneider\}@pnnl.gov}
\vspace{-0.25in}}


\maketitle

\begin{abstract}
As modern engineering systems grow in complexity, {attitudes toward} a modular design approach become increasingly more favorable. A key challenge to a modular design approach is the {certification} of robust stability under uncertainties in the rest of the network. In this paper, we consider the problem of identifying the parametric region, which guarantees stability of the connected module in the robust sense under uncertainties. We derive the conditions under which the robust stability of the connected module is guaranteed for some values of the design parameters, and present a sum-of-squares (SOS) optimization-based algorithm to identify such a parametric region for polynomial systems. Using the example of an inverter-based microgrid, we show how this parametric region {changes} with variations in the level of uncertainties in the network.
\end{abstract}



\IEEEpeerreviewmaketitle

\section{Introduction}

With the growing complexity of modern engineering systems, {attitudes toward} reconfigurability and modular design approaches are gaining popularity. Plug-and-play design approaches have drawn attention {for use} in cyber-physical networks, power grids, biological networks, and process control systems \cite{baheti2011cyber,farhangi2010path,huang2011future,litcofsky2012iterative,bendtsen2013plug}. In the context of microgrids, and power systems in general, the plug-and-play design approach is particularly attractive because of the involvement of various stakeholders (not all resources/equipment on the network are owned by the same utility). A hierarchical design is often preferred, where a network-level assessment of the operational conditions sets certain interconnection guidelines (from the dynamic security and economic considerations) to which individual resource owners (or resource aggregators) adhere when plugging in their resource to the network \cite{nehrir2011review,planas2013general,lasseter2011smart}. As such, {a key challenge for a successful plug-and-play operation} is the identification of the design parameter space that certifies robust stability under various operational conditions of the network.

Unlike bulk power systems, which have adequate rotational inertia to naturally stabilize fluctuations in the network, the dynamic security of low-inertia microgrids needs to be specifically ensured via design \cite{Xu:2018,mashayekh2018security}. Identification of droop-coefficients for stability certification of inverter-based microgrids have been investigated in recent works \cite{Schiffer:2014,Vorobev:2018}. A centralized approach of identifying the droop-coefficients for a lossless microgrid was adopted in \cite{Schiffer:2014}, while conditions on droop-coefficients were derived in a distributed approach for small-signal stability in \cite{Vorobev:2018}. However, low-to-medium voltage microgrids typically have significant line resistance-to-reactance ratios, and often operate in a nonlinear regime due to fluctuations from renewable generation, rendering the aforementioned approaches inapplicable. 

Lyapunov function methods have been widely used in the context of nonlinear systems stability certification \cite{Lyapunov:1892,Khalil:1996}. Extension of the theory to robust stability problems under uncertainties as well as parametric stability analysis have been proposed \cite{blanchini1999set,Gahinet:1996}. More recent works have used advanced computational techniques, such as sum-of-squares (SOS) algorithms, for parametric stability analysis using Lyapunov functions \cite{Anderson:2015,Wloszek:2003,Parrilo:2000,Tan:2006,Anghel:2013}. Lyapunov-based methods have been applied to robust stability analysis and control problems in power grids \cite{vu2017framework,wang1998robust}. Chebyshev minimax formulation has been used for identifying the parametric stability region for linear systems (with Lur'e-type nonlinearity) \cite{siljak1989parameter}. The construction of the design parameter space that ensures robust stability of nonlinear networks, though, still remains a challenge. 

The main contributions of this article are - 1) the theoretical construction of robust stability certificates in the design parameter space for nonlinear systems, under exogenous time-varying but bounded uncertainties; and 2) an algorithmic approach to identifying the largest robust stability region in the parametric space for polynomial networks. Numerical illustrations are provided in the context of identifying droop-coefficient values for robustly stable plug-and-play design of inverter-based microgrids. The rest of this article is structured as follows: Section\,\ref{S:prelim} provides a background of the relevant theoretical and computational methods; Section\,\ref{S:problem} presents a description of the microgrid example and the problem formulation; Sections\,\ref{S:theory} and \ref{S:algo} describe the theoretical and algorithmic approach; while a numerical example is presented in Section\,\ref{S:example}. We conclude this article in Section\,\ref{S:concl}. Throughout the text, $\left|\,\cdot\,\right|$ will be used to denote both the $\mathcal{L}_2$-norm of a vector and the absolute value of a scalar; while $\nabla_x$ denotes the gradient (of a function) with respect to $x$\,.

\section{Preliminaries}\label{S:prelim}

\subsection{Stability Analysis: Lyapunov Functions}

Consider a nonlinear system of the form:
\begin{align}\label{E:f}
\mathcal{S}:\quad\dot{x}  = f(x)\,,\quad x\in\mathcal{X}\subseteq\mathbb{R}^n
\end{align}
{The equilbrium point of interest is shifted to the origin ($0\in\mathcal{X}$) and $f$ is assumed to be locally Lipschitz in $\mathcal{X}$}. 
%
The equilibrium at origin is said to be locally asymptotically stable if 1) for every $\nu>0$ there exists an $\epsilon>0$ such that $|x(t)|\leq\nu\,\forall t\geq 0$ for every $|x(0)|\leq\epsilon$\,, and 2) $\lim_{t\rightarrow\infty}\left\vert x(t)\right\vert =0\,$ for every $x(0)$ in $\mathcal{X}$\,. Lyapunov's stability conditions state\cite{Lyapunov:1892,Khalil:1996} :
%
\begin{theorem}\label{T:Lyap}
Existence of a continuously differentiable radially unbounded positive definite function $\Psi\!:\!\mathcal{X}\!\rightarrow\!\mathbb{R}_{\geq 0}$ (Lyapunov function) {where $\nabla_x{\Psi}^T\!f(x)$ is negative definite in $\mathcal{X}$ guarantees asymptotic} stability of the origin. 
\end{theorem}


\subsection{Parametric Lyapunov Analysis}
Consider a dynamical system in a parametric form:
\mysubeq{E:flam}{
\dot{x} &= f(x,\lambda)\,,\,\quad x\in\mathcal{X}\,,\,\lambda\in\Lambda\\
\text{where}\,~\mathcal{X}&:=\lbrace x\left| \,a_i(x)\geq 0\,,\,i\in\lbrace 1,\dots,p\rbrace\right.\rbrace.
}
{$\lambda\in\mathbb{R}^l$ is an $l$-dimensional vector of the design parameters and $f$ is assumed to be locally Lipschitz continuous in $\mathcal{X}$.}
%
Assume that over the range of possible values of the parameters, the equilibrium point of interest always remains at the origin, i.e.,
%
$f(0,\lambda) = 0~\forall \lambda\in\Lambda$\,.
%
%
Stability of such systems can be analyzed in a similar treatment to that of the \textit{absolute stability} problem \cite{Khalil:1996,Anderson:2015}. We argue that the origin of the parametric system \eqref{E:flam} is locally asymptotically stable in $\mathcal{X}$ if there exists a continuously differentiable parametric Lyapunov function $\Psi:\mathcal{X}\times\Lambda\rightarrow\mathbb{R}_{\geq 0}$ satisfying 
\begin{subequations}\label{E:Vun}
\begin{align}
\Psi(x,\lambda)&\geq \phi_1(x)\quad \forall (x,\lambda)\in\mathcal{X}\times\Lambda\\
\nabla_x \Psi^Tf(x,\lambda)&\leq-\phi_2(x)\quad\forall (x,\lambda)\in\mathcal{X}\times\Lambda
\end{align}\end{subequations}
for some positive definite functions $\phi_1(\cdot)$ and $\phi_2(\cdot)$\,. 
%
The conditions can be extended to the situations when the equilibrium point depends on the values of the parameter. For example, if the equilibrium point of interest is an explicit function of the parameter, $x_0(\lambda)$\,, then the above argument holds after shifting of the state variables $\tilde{x}=x-x_0(\lambda)$\,. 
%

\subsection{Sum-of-Squares Optimization}
Relatively recent studies have explored how SOS-based methods can be utilized to find Lyapunov functions by restricting the search space to SOS polynomials \cite{Wloszek:2003,Parrilo:2000,Tan:2006,Anghel:2013}. 
Let us denote by $\mathbb{R}\left[x\right]$ the ring of all polynomials in $x\in\mathbb{R}^n$. 
A multivariate polynomial $p\in\mathbb{R}\left[x\right],~x\in\mathbb{R}^n$, is an SOS if there exist some polynomial functions $h_i(x), i = 1\ldots s$ such that 
$p(x) = \sum_{i=1}^s h_i^2(x)$.
We denote the ring of all SOS polynomials in $x$ by $\Sigma[x]$.
Whether or not a given polynomial is an SOS is a semi-definite problem which can be solved with SOSTOOLS, a MATLAB$^\text{\textregistered}$ toolbox \cite{sostools13}, along with a semi-definite programming solver such as SeDuMi \cite{Sturm:1999}. 
An important result from algebraic geometry, called Putinar's Positivstellensatz theorem \cite{Putinar:1993,Lasserre:2009}, helps in translating conditions such as in \eqref{E:Vun} into SOS feasibility problems. 
\begin{theorem}\label{T:Putinar}
Let $\mathcal{K}\!\!=\! \left\lbrace x\in\mathbb{R}^n\left\vert\, k_1(x) \geq 0\,, \dots , k_m(x)\geq 0\!\right.\right\rbrace$ be a compact set, where $k_j$ are polynomials. Define $k_0=1\,.$ Suppose there exists a $\mu\!\in\! \left\lbrace {\sum}_{j=0}^m\sigma_jk_j \left\vert\, \sigma_j \!\in\! \Sigma[x]\,\forall j \right. \right\rbrace$ such that $\left\lbrace \left. x\in\mathbb{R}^n \right\vert\, \mu(x)\geq 0 \right\rbrace$ is compact. Then,  
\begin{align*}
p(x)\!>\!0~\forall x\!\in\!\mathcal{K}
\!\implies\! p \!\in\! \left\lbrace {\sum}_{j=0}^m\sigma_jk_j \left\vert\, \sigma_j \!\in\! \Sigma[x]\,\forall j \right. \right\rbrace\!.
\end{align*}
\end{theorem}
Using Theorem\,\ref{T:Putinar}, we can translate the problem of checking that $p\!>\!0$ on $\mathcal{K}$ into an SOS feasibility problem where we seek the SOS polynomials $\sigma_0\,,\,\sigma_j\,\forall j$ such that $p\!-\!\sum_j\sigma_j k_j$ is {SOS.
Note that any} equality constraint $k_i(x)\!=\!0$ can be expressed as two inequalities $k_i(x)\!\geq 0$ and $k_i(x)\!\leq\! 0$. In many cases, especially for the $k_i\,\forall i$ used throughout this work, a $\mu$ satisfying the conditions in Theorem\,\ref{T:Putinar} is guaranteed to exist (see \cite{Lasserre:2009}), and need not be searched for.
%

\section{Problem Description}\label{S:problem}

\subsection{Motivational Example: Microgrids}
Design of a networked microgrid involves solving an optimization problem that ensures operational reliability (e.g., transient stability) while achieving certain economic goals \cite{Barnes:2017}. Typically this translates to identifying the largest region in the space of design parameters that certify stability of the system under a set of uncertainties. Consider the case of droop-controlled inverters \cite{Coelho:2002,Schiffer:2014}:
\mysubeq{E:droop}{
\dot{\theta}_i & = \omega_i\,,\\
\tau_i\dot{\omega}_i & = -\omega_i + \lambda_i^p \left(P_i^d-P_i\right)\\
\tau_i\dot{v}_i & = v_i^d-v_i + \lambda_i^q \left(Q_i^d-Q_i\right)
}
where $\lambda_i^p>0$ and $\lambda_i^q>0$ are the droop-coefficients associated with the active power vs. frequency and the reactive power vs. voltage droop curves, respectively; $\tau_i$ is the time-constant of a low-pass filter used for the active and reactive power measurements; $\theta_i\,,\,\omega_i$ and $v_i$ are, respectively, the phase angle, frequency, and voltage magnitude; $v^d_i,\,P_i^d$ and $Q^d_i$ are the nominal values of the voltage magnitude, active power, and reactive power, respectively. $P_i$ and $Q_i$ are, respectively, the active and reactive power injected into the network, related to the neighboring bus voltage phase angle and magnitude as:
\mysubeq{E:PQ}{
P_i &= v_i{\sum}_{k\in\mathcal{N}_i} v_k\left(G_{i,k}\cos\theta_{i,k} + B_{i,k}\sin\theta_{i,k}\right)\\
Q_i &= v_i{\sum}_{k\in\mathcal{N}_i} v_k\left(G_{i,k}\sin\theta_{i,k} - B_{i,k}\cos\theta_{i,k}\right)
}
where $\theta_{i,k}=\theta_i-\theta_k$\,, and $\mathcal{N}_i$ is the set of neighbor nodes; and $G_{i,k}$ and $B_{i,k}$ are respectively the transfer conductance and susceptance values of the line connecting the nodes $i$ and $k$\,. {Considering the droop-coefficients as the design parameters, the goal of this work is the algorithmic identification of the design space that ensures robust stability of the microgrid. Note that the particular choice of droop-coefficients as design parameters is for illustrative purpose only, while the proposed algorithm is generalizable to other choices of design parameters (such as line parameters and dispatched power set-points).} 

The nominal (or desired) equilibrium is attained when 
\begin{align*}
\forall i:\quad P_i=P_i^d,\,Q_i=Q_i^d,\,\omega_i=0,\,v_i=v_i^d\,.
\end{align*} 
In a plug-and-play {operation}, it is important that the design parameters are chosen to ensure robust stability of the (possibly) \textit{time-varying} equilibrium point of the connected inverter under bounded uncertainties in the (rest of the) network. 
Moreover, an additional constraint that needs to be enforced through the choice of design parameters is that the equilibrium point under uncertainties should stay \textit{close} to the nominally desired equilibrium of $(\omega_i,v_i)=(0,v^d_i)$\,. This ensures that even under uncertainties, the operating conditions remain \textit{acceptable}.

{After introducing} the following variables:
\begin{align}\label{E:droop_disturbance}
\delta_{1,i,k}:=v_k\cos\theta_{i,k}\,\text{ and }\,\delta_{2,i,k} := v_k\sin\theta_{i,k}\,,
\end{align}
the inverter dynamics \eqref{E:droop}-\eqref{E:PQ} can be reformulated in the polynomial form as follows:
\mysubeq{E:droop_poly}{
\!\!\!\tau_i\dot{\omega}_i & = -\omega_i \!+\! \lambda_i^p \left[P_i^d\!-\!v_i\!\!\!\sum_{k\in\mathcal{N}_i} \!(G_{i,k}\delta_{1,i,k} \!+\! B_{i,k}\delta_{2,i,k})\right]\!\!\\
\!\!\!\tau_i\dot{v}_i & = v_i^d\!-\!v_i \!+\! \lambda_i^q \left[Q_i^d\!-\!v_i\!\!\!\sum_{k\in\mathcal{N}_i} \!(G_{i,k}\delta_{2,i,k} \!-\! B_{i,k}\delta_{1,i,k})\right]\!\!
}
In the reformulation, the phase angle dynamics are dropped, since the phase angle differences (represented in $\delta_{1,i,k}$ and $\delta_{2,i,k}$) are sufficient to model the power flow across networks.

\subsection{Problem Formulation}
Consider an uncertain polynomial dynamical system which is represented in a parametric form as follows:
\mysubeq{E:flamd}{
\mathcal{S}[\lambda,\delta]:~\dot{x}(t) &= f(x(t),\lambda,\delta(t))\,,\,~\left\lbrace\begin{array}{rl}
x(t)\!\!&\!\!\in\mathcal{X}\,,\\
\lambda\!\!&\!\!\in\Lambda\,,\\
\delta(t)\!\!&\!\!\in\mathcal{D}
\end{array} \right.\\
\text{where}\,~\mathcal{X}&:=\lbrace x\left| \,a_i(x)\geq 0\,,\,i\in\lbrace 1,\dots,p\rbrace\right.\rbrace\,,\\
\mathcal{D}&:=\lbrace \delta\left| \,b_i(\delta)\geq 0\,,\,i\in\lbrace 1,\dots,q\rbrace\right.\rbrace
}
where $\delta(t)\in\mathbb{R}^d$ denote a $d$-dimensional vector of uncertain and (possibly) time-varying exogenous parameters, which lie in a semi-algebraic domain $\mathcal{D}$\,; $f,\,a_i,\,b_i$ are polynomials. For notational simplicity, we will henceforth drop the time parameter $t$ from the argument of $x$ and $\delta$\,, whenever obvious\,. Without any loss of generality, we assume that $0\in\mathcal{D}$, and that $x=0$ is an equilibrium of the system when $\delta=0$\,, i.e.,
\begin{align}\label{E:nom_eq}
f(0,\lambda,0) = 0\quad\lambda\in\Lambda\,.
\end{align}
Moreover, when $\delta\neq 0$\,, the equilibrium point of interest, $x_0(\lambda,\delta)$\,, is defined uniquely in the domain $(x,\lambda,\delta)\in\mathcal{X}\times\Lambda\times\mathcal{D}$ by the relationship:
\begin{align}\label{E:pert_eq}
x_0(\lambda,\delta):=\lbrace x\in\mathcal{X}\,|\,f(x,\lambda,\delta) = 0\rbrace~\forall (\lambda,\delta)\in\Lambda\times\mathcal{D}.
\end{align} 
\begin{remark}
We assume the explicit functional form $x_0(\lambda,\delta)$ to be available. Future work will address the issues when this relationship is implicit. Also note that the condition \eqref{E:nom_eq} holds when droop coefficients are chosen as the design parameters values. Future efforts will consider relaxing that condition.
\end{remark}

The problem we are interested in is identifying a set of possible values of the design parameter $\lambda$ that ensures the robust stability of the system \eqref{E:flamd} under bounded uncertainties, i.e., find the set $\widehat{\Lambda}\subseteq{\Lambda}$ such that the following hold:
\begin{enumerate}
\item the equilibrium point of interest, $x_0(\lambda,\delta)$, 
remain within an \textit{acceptable} region $\mathcal{X}_0\subseteq\mathcal{X}$ ($0\in\mathcal{X}_0$) for every uncertainty $\delta\in\mathcal{D}$ and for every design parameter $\lambda\in\widehat{\Lambda}$\,;
\item the \textit{locally} asymptotic stability of the equilibrium point $x_0(\lambda,\delta)$ of the uncertain system $\mathcal{S}[\lambda,\delta]$ in \eqref{E:flamd} is guaranteed for every $\delta\in\mathcal{D}$ and for every $\lambda\in\widehat{\Lambda}$\,.
\end{enumerate}


\section{Theoretical Construction}\label{S:theory}
In this section, we discuss the theoretical development regarding robust stability of the connected module over some parameter range, under bounded uncertainties.
\begin{assumption}\label{AS:closeness}
The system \eqref{E:flamd} admits a unique equilibrium point $x_0(\lambda,\delta)$ inside the domain $\lbrace x|\,|x|\!\leq\!\Delta\rbrace\!\subset\!\mathcal{X}$, i.e., 
\begin{align*}
(\lambda,\delta)\in\Lambda\times\mathcal{D}\implies\exists\, x_0(\lambda,\delta)\in\lbrace x|\,|x|\!\leq\!\Delta\rbrace\!\subset\!\mathcal{X}~\text{s.t.}~\eqref{E:pert_eq}
\end{align*}
\end{assumption}
Note that the value of $\Delta$ depends not only on the uncertainties, but also on the parameter values. Given some parameter value, $\Delta$ decreases as the uncertainty level goes down. On the other hand, given a range of uncertainties, we can choose the range of parameter values to lower $\Delta$\,.
\begin{assumption}\label{AS:lyapunov}
The system \eqref{E:flamd} admits a parametric Lyapunov function $\Psi(x,\lambda)$ satisfying the following:
\begin{align*}
\forall (x,\lambda,\delta)\in\mathcal{X}\!\times\!\Lambda\!\times\!\mathcal{D}:\quad \Psi(x,\lambda)\geq \phi_1(x\!-\!x_0(\lambda,\delta))\\
\nabla_x\Psi^T\!f(x,\lambda,\delta)\leq -\phi_2(x\!-\!x_0(\lambda,\delta))
\end{align*}
where the equilibrium of interest $x_0(\lambda,\delta)$ satisfies Assumption\,\ref{AS:closeness}; and $\phi_{1,2}(\cdot)$ are positive definite functions.
\end{assumption}

Let us define
\mysubeq{}{
\Gamma &:=\max\left\lbrace \gamma\,|\,|x|\leq \gamma\implies x\in\mathcal{X}\right\rbrace\\
\Gamma_\Psi(\lambda) &:=\max\left\lbrace \gamma\,|\,\Psi(x,\lambda)\leq \gamma\implies x\in\mathcal{X}\right\rbrace
} 
i.e., $\Gamma$ is the largest level-set of the $\mathcal{L}_2$-norm of the state $x$ contained within $\mathcal{X}$\,, while $\Gamma_\Psi(\lambda)$ is the maximum level-set of $\Psi(x,\lambda)$ contained within $\mathcal{X}$. Note that,
\begin{align*}
\Delta<\Gamma\,.
\end{align*}

\begin{proposition}\label{P:bounded} (Boundedness)
Let us define the following:
\mysubeq{E:bounds}{
\zeta^*(\lambda)&=\min\left\lbrace \zeta\,\left| \begin{array}{c} |x-x_0(\lambda,\delta)|\leq\Delta,\,\delta\in\mathcal{D}\\
\implies\Psi(x,\lambda)\leq\zeta\end{array}\right.\right\rbrace\\
\nu^*(\lambda)&=\min\left\lbrace \nu\,\left| \begin{array}{c} \Psi(x,\lambda)\leq\zeta^*,\,\delta\in\mathcal{D}\\
\implies |x-x_0(\lambda,\delta)|\leq\nu\end{array}\right.\right\rbrace.
}
For sufficiently weak uncertainties satisfying 
\begin{align}
\Delta<\Gamma-\nu^*(\lambda)\,,\text{ and }\,\zeta^*(\lambda)<\Gamma_\Psi(\lambda)\,, 
\end{align}
there exists a $\xi>0$ for every $\nu\in[\nu^*(\lambda)+\Delta,\Gamma]$ such that $|x(0)|\leq\xi$ implies $|x(t)|\leq\nu$ for all $t\geq 0$\,.  
\end{proposition}
\begin{proof}
From Assumption\,\ref{AS:lyapunov}, we have 
\begin{align*}
\Psi(x(t),\lambda)-\Psi(x(0),\lambda)\leq-\int_0^t\phi_2(x(\tau)-x_0(\lambda,\delta))\,d\tau
\end{align*}
i.e., $\Psi(x,\lambda)$ is non-increasing in $x$ along the trajectories of the system \eqref{E:flamd}. Note that when $x(0)=0$\,, $|x(0)-x_0(\lambda,\delta)|\leq\Delta$\,. From \eqref{E:bounds} it follows immediately that $x(0)=0$ implies $\Psi(x(0),\lambda)\leq\zeta^*$ which implies $\Psi(x(t),\lambda)\leq\zeta^*$ for all $t\geq 0$\,. Applying \eqref{E:bounds} again, we have $|x(t)-x_0(\lambda,\delta)|\leq\nu^*$, such that 
\begin{align*}
\forall t\geq 0: ~|x(t)|\leq|x(t)-x_0(\lambda,\delta)|+|x_0(\lambda,\delta)|\leq\nu^*+\Delta\,.
\end{align*}
For sufficiently weak uncertainties satisfying $\Delta<\Gamma-\nu^*(\lambda)$\,, we have that $x(0)=0$ implies $|x(t)|<\Gamma$ for all $t\geq 0$\,. 

Now, for every $\nu\in(\nu^*(\lambda)+\Delta,\Gamma]$\,, we have 
\begin{align*}
|x(t)-x_0(\lambda,\delta)|\leq\nu-\Delta\implies |x(t)|\leq\nu\,.
\end{align*}
Moreover, because $\phi_1(\cdot)$ is a radially unbounded and positive definite function bounding $\Psi(x,\lambda)$ from below, for every such $\nu-\Delta>\nu^*$\,, we have a $\zeta>\zeta^*$ such that 
\begin{align*}
\Psi(x(t),\lambda)\leq\zeta\implies |x(t)-x_0(\lambda,\delta)|\leq\nu-\Delta\,.
\end{align*}
Since $\Psi(x,\lambda)$ is non-increasing in $x$ along system trajectories, and since $\phi_1(\cdot)$ is positive definite, there exists a $\xi>0$ for every $\zeta>\zeta^*$ such that 
\begin{align*}
|x(0)|\leq\xi&\implies\Psi(x(0),\lambda)\leq\zeta\\
&\implies\Psi(x(t),\lambda)\leq\zeta\implies|x(t)|\leq\nu\,.
\end{align*}
This completes the proof. {Fig.\,\ref{F:illustrate} illustrates the different level-sets used in the derivation.}\hfill\qed\end{proof}
%
\begin{figure}[thpb]
\centering
\includegraphics[scale=0.35]{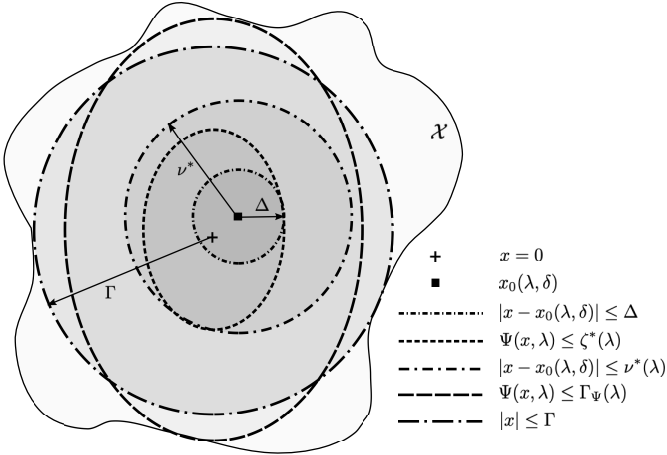}
\caption{{An illustration of the different level-sets.}}
\label{F:illustrate}
\end{figure}
\begin{proposition}\label{P:converge}(Convergence)
Let us define the following\footnote{Note that, by construction, $\mu^*(\lambda)<\Gamma$\,.}:
\begin{align}
\mu^*(\lambda):=\max\left\lbrace \mu\,\left| \begin{array}{c} |x-x_0(\lambda,\delta)|\leq\mu,\,\delta\in\mathcal{D}\\
\implies\Psi(x,\lambda)\leq\Gamma_\Psi(\lambda)\end{array}\right.\right\rbrace.
\end{align}
For sufficiently weak uncertainties satisfying $\Delta<\mu^*(\lambda)/2$\,, there exists a finite time $T(\mu,\epsilon)$ for every $\mu\in[\Delta,\mu^*(\lambda)-\Delta]$ and $\epsilon\in(0,\Gamma-\Delta]$ such that $|x(t)|\leq\epsilon+\Delta$ for all $t\geq T(\mu,\epsilon)$ for every $|x(0)|\leq\mu$\,.
\end{proposition}
\begin{proof}
For every $\mu$ such that $|x(0)|\leq\mu$\,, we have $|x(0)-x_0(\lambda,\delta)|\leq\mu+\Delta$\,. Since $\phi_1(\cdot)$ is positive definite, there exists a $\rho^*$ such that 
\begin{align*}
|x(0)-x_0(\lambda,\delta)|\leq\mu+\Delta\implies\Psi(x(0),\lambda)\leq\rho^*\,.
\end{align*}
For sufficiently weak uncertainties satisfying $\Delta<\mu^*(\lambda)/2$\,, we have $\rho^*\leq\Gamma_\Psi(\lambda)$ for every $\mu\in[\Delta,\mu^*(\lambda)-\Delta]$\,.

Since $\phi_1(\cdot)$ is radially unbounded and positive definite, there exists a $\rho_*\in(0,\rho^*)$ for every $\epsilon\in(0,\Gamma-\Delta]$ such that
\begin{align*}
\Psi(x(t),\lambda)\leq\rho_*&\implies |x(t)-x_0(\lambda,\delta)|\leq\epsilon\\
&\implies|x(t)|\leq\epsilon+\Delta\,.
\end{align*}
Let us define: 
\begin{align*}
\kappa(\lambda):=\min\left\lbrace \phi_2(x-x_0(\lambda,\delta))\,|\,\Psi(x,\lambda)\in[\rho_*,\rho^*]\,,\,\delta\in\mathcal{D}\right\rbrace\!.
\end{align*}
Choosing $T(\mu,\epsilon)=(\rho^*-\rho_*)/\kappa$, we can show that:
\begin{align*}
\forall t\geq T(\mu,\epsilon):~\rho^*-\Psi(x(t),\lambda)&\geq\Psi(x(0),\lambda)-\Psi(x(t),\lambda)\\
&\geq \kappa\,t\geq \kappa\,T(\mu,\epsilon)\geq \rho^*-\rho_*\\
\implies\quad  \Psi(x(t),\lambda)&\leq\rho_*\,.
\end{align*}
This completes the proof.\hfill\hfill\qed\end{proof}

\begin{theorem}
(Main Result) Suppose Assumptions\,\ref{AS:closeness} \& \ref{AS:lyapunov} hold, and the uncertainties are sufficiently weak such that
\begin{align*}
\begin{array}{c}\Delta<\min\left(\Gamma-\nu^*(\lambda),\,\mu^*(\lambda)/2\right)\,,\\
\text{and }~\,\zeta^*(\lambda)<\Gamma_\Psi(\lambda)\,,\end{array}
\end{align*}
then the system $\mathcal{S}[\lambda,\delta]$ in \eqref{E:flamd} satisfies the following boundedness and uniform asymptotic convergences properties: there exists a $\xi>0$ for every $\nu\in[\nu^*(\lambda)+\Delta,\Gamma]$ such that $|x(0)|\leq\xi$ implies $|x(t)|\leq\nu$ for all $t\geq 0$\,, and
\begin{align*}
\forall \mu\in[\Delta,\mu^*(\lambda)-\Delta]:~|x(0)|\leq\mu\implies\lim_{t\rightarrow\infty}|x(t)|\leq\Delta\,.
\end{align*}
\end{theorem}
\begin{proof}
Follows from Propositions\,\ref{P:bounded} and \ref{P:converge}.
\hfill\hfill\qed\end{proof}

\section{Algorithmic Procedure}\label{S:algo}

In this section we present an algorithmic procedure to compute the largest parameter set with certified robust stability. Without any loss of generality, let us assume that $0\in\Lambda$\,,\footnote{This can be achieved by defining new parameters $\tilde{\lambda}=\lambda-\lambda^{\min}$.} and that the origin is a locally asymptotically stable equilibrium point of the nominal (unperturbed) system $\mathcal{S}(0,0)$\,. In the rest of this article, we will restrict ourselves to the identification of the region of design parameter space in the form of 
\begin{align}\label{E:lambda_hat}
\widehat{\Lambda}(\beta):=\left\lbrace \lambda\in\mathbb{R}^l\left|\,G\,\lambda\leq \beta\,h\right.\right\rbrace\,,
\end{align}
where $\beta\geq 0$ is a scalar, $h=[h_i]$ is an $m$-dimensional vector of non-negative scalars, for some $m\geq 1$, i.e. $h_i\geq0\,\forall i\in\lbrace 1,2,\dots,m\rbrace$, and $G=[g_{ij}]$ is an $m\times l$ matrix. Note that $0\in\widehat{\Lambda}(0)$\,. Moreover,
\begin{align*}
\widehat{\Lambda}(\beta_1)\subseteq\widehat{\Lambda}(\beta_2)\quad\forall \beta_2\geq\beta_1\geq 0\,.
\end{align*}

We are interested in solving the following problem:
\mysubeq{E:opt}{
\max_{\Psi(x,\lambda)}&\quad\beta\\
\text{subject to,}&\quad\forall (x,\lambda,\delta)\in\mathcal{X}\!\times\!\widehat{\Lambda}(\beta)\!\times\!\mathcal{D}:\nonumber\\
&\quad\Psi(x,\lambda)\geq \varepsilon_1\left|x\!-\!x_0(\lambda,\delta)\right|^2\\
&\quad\nabla_x\Psi^T\!f(x,\lambda,\delta)\leq -\varepsilon_2\left|x\!-\!x_0(\lambda,\delta)\right|^2\\
&\quad|x_0(\lambda,\delta)|^2\leq\Delta^2
}
where $\mathcal{X}$ and $\mathcal{D}$ are semi-algebraic domains defined in \eqref{E:flamd}, while $\varepsilon_{1,2}$ are small positive scalars. The first two constraints are the Lyapunov conditions, while the third constraint is to make sure that the equilibrium point under uncertainties do not move far from the nominal (desired) equilibrium point at the origin. Using Theorem\,\ref{T:Putinar}, the above problem can be recast into an SOS optimization problem as follows:
\begin{align}\label{E:opt_sos}
\underset{\Psi(x,\lambda),\lbrace s^{k1}_i\rbrace,\lbrace s^{k2}_i\rbrace,\lbrace s^{k3}_i\rbrace\,\forall k\in\lbrace 1,2,3\rbrace}{\max}\quad\beta\qquad\\
\text{subject to:}\qquad\qquad\qquad\qquad\qquad\qquad\qquad\qquad\qquad\notag\\
\!\!\Psi(x,\lambda)- \varepsilon_1\left|x\!-\!x_0(\lambda,\delta)\right|^2-\sum_{i=1}^ps^{11}_ia_i(x)\notag\\
\!\!+\sum_{i=1}^ms^{12}_i(\sum_{j=1}^lg_{ij}\lambda_j-\beta h_i)-\sum_{i=1}^qs^{13}_ib_i(\delta)\in\Sigma[x,\lambda,\delta],\notag\\
-\nabla_x\Psi^T\!f(x,\lambda,\delta) -\varepsilon_2\left|x\!-\!x_0(\lambda,\delta)\right|^2-\sum_{i=1}^ps^{21}_ia_i(x)\notag\\
\!\!+\sum_{i=1}^ms^{22}_i(\sum_{j=1}^lg_{ij}\lambda_j-\beta h_i)-\sum_{i=1}^qs^{23}_ib_i(\delta)\in\Sigma[x,\lambda,\delta],\notag\\
\Delta^2 -\left|x_0(\lambda,\delta)\right|^2\notag\\
\!\!+\sum_{i=1}^ms^{32}_i(\sum_{j=1}^lg_{ij}\lambda_j-\beta h_i)-\sum_{i=1}^qs^{33}_ib_i(\delta)\in\Sigma[x,\lambda,\delta],\notag
\end{align}
where $\lbrace s^{k1}_i\rbrace\,\forall k\in\lbrace 1,2\rbrace,\lbrace s^{k2}_i\rbrace\,\forall k\in\lbrace 1,2,3\rbrace,\lbrace s^{k3}_i\rbrace\,\forall k\in\lbrace 1,2,3\rbrace$ are multi-variate SOS polynomials from the ring $\Sigma[x,\lambda,\delta]$\,. There are two challenges to solving this problem: 1) the explicit functional form of $x_0(\lambda,\delta)$ may not be available in polynomial form (or at all); and 2) the decision variables are in bilinear form, such as the terms $s^{12}_i\beta,\,s^{22}_i\beta,$ and $s^{32}_i\beta$\,. 
{The first challenge can be resolved by obtaining \textit{sufficiently close} polynomial approximation of $x_0(\lambda,\delta)$ via Taylor series expansion around $(\lambda,\delta)=(0,0)$ (or, by polynomial recasting techniques \cite{Antonis:2005}). The second challenge is resolved by reformulating \eqref{E:opt_sos} as an iterative feasibility problem while applying a bisection-search algorithm for the maximum value of $\beta$.}

\section{Example: Inverter-Based Microgrid}\label{S:example}
%
%
We consider a modified version of the CERTS microgrid network described in \cite{lasseter2011certs} as an example. Disconnecting the utility, we replace the substation by a droop-controlled inverter, with two other inverters placed alongside load banks 3 and 5 (no inverters at load banks 4 and 6)\,. Nominal operating point (equilbrium) of the network was obtained by solving the steady-state power-flow equations \eqref{E:PQ}. 
A disturbance set was created by allowing the uncertain parameters to vary within some limits around their nominal values (denoted by superscript `nom') in the form of:
{\begin{align*}
\forall i,\,\forall k\in\mathcal{N}_i:~\left|\frac{\delta_{1,i,k}-\delta_{1,i,k}^{\text{nom}}}{\delta_{1,i,k}^{\text{nom}}}\right|\leq \alpha\,,\,\left|\frac{\delta_{2,i,k}-\delta_{2,i,k}^{\text{nom}}}{\delta_{2,i,k}^{\text{nom}}}\right|\leq \alpha
\end{align*}
w}here the value of $\alpha>0$ denotes different levels of uncertainties. The design parameter set for the droop-coefficients was chosen to be of the form \eqref{E:lambda_hat} with the affine constraints
\begin{align*}
\lambda^p_i\in(0,\beta]\,\text{ and }\,\lambda^p_i\in(0,0.2\beta]\,.
\end{align*}
Small positive scalars were used as the minimum values for the droop-coefficients, as per the typical norm on grid operations. Notice that when $\lambda^p_i\ll 1$ and $\lambda^q_i\ll 1$ the inverter voltage and frequency become \textit{stiff}, not adjusting with network conditions, which is an unfavorable scenario from the network resiliency perspective.
The perturbed equilibrium point is desired to remain within some domain of the form:
\begin{align}\label{E:eq_dom}
\left\lbrace (\omega_i,v_i)\,\left|\,\left(\frac{\omega_i}{\omega^{\max}}\right)^2+\left(\frac{v_i-v_i^d}{\Delta v^{\max}}\right)^2\leq c\right.\right\rbrace
\end{align}
where $\omega^{\max}$ was set to $0.7\,$Hz, and $\Delta v^{\max}$ to $0.2\,$p.u.\,. The value of $c$ was varied to investigate different uncertainty scenarios. Note that the constraint defining the domain \eqref{E:eq_dom} is equivalent to the third constraint in \eqref{E:opt}, albeit after scaling and shifting. The choice of $c$ influences the possible set of design parameter values (with smaller values yielding narrower design space).
\begin{figure*}[thpb]
\centering
\subfigure[parametric set for inverter \#1]{
\includegraphics[scale=0.32]{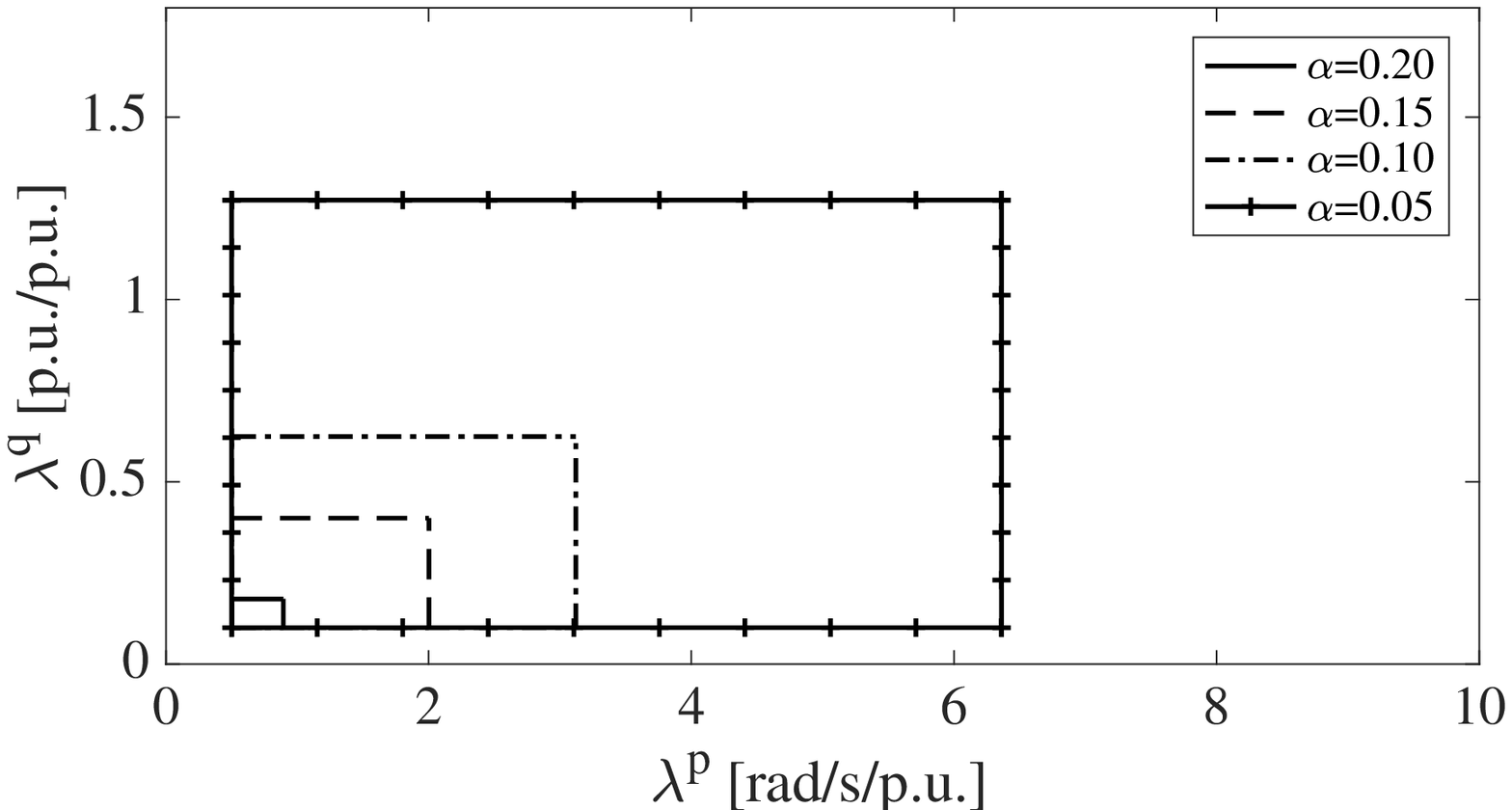}\label{F:inverter1}
}
\hspace{-0.4in}
\subfigure[parametric set for inverter \#2]{
\includegraphics[scale=0.32]{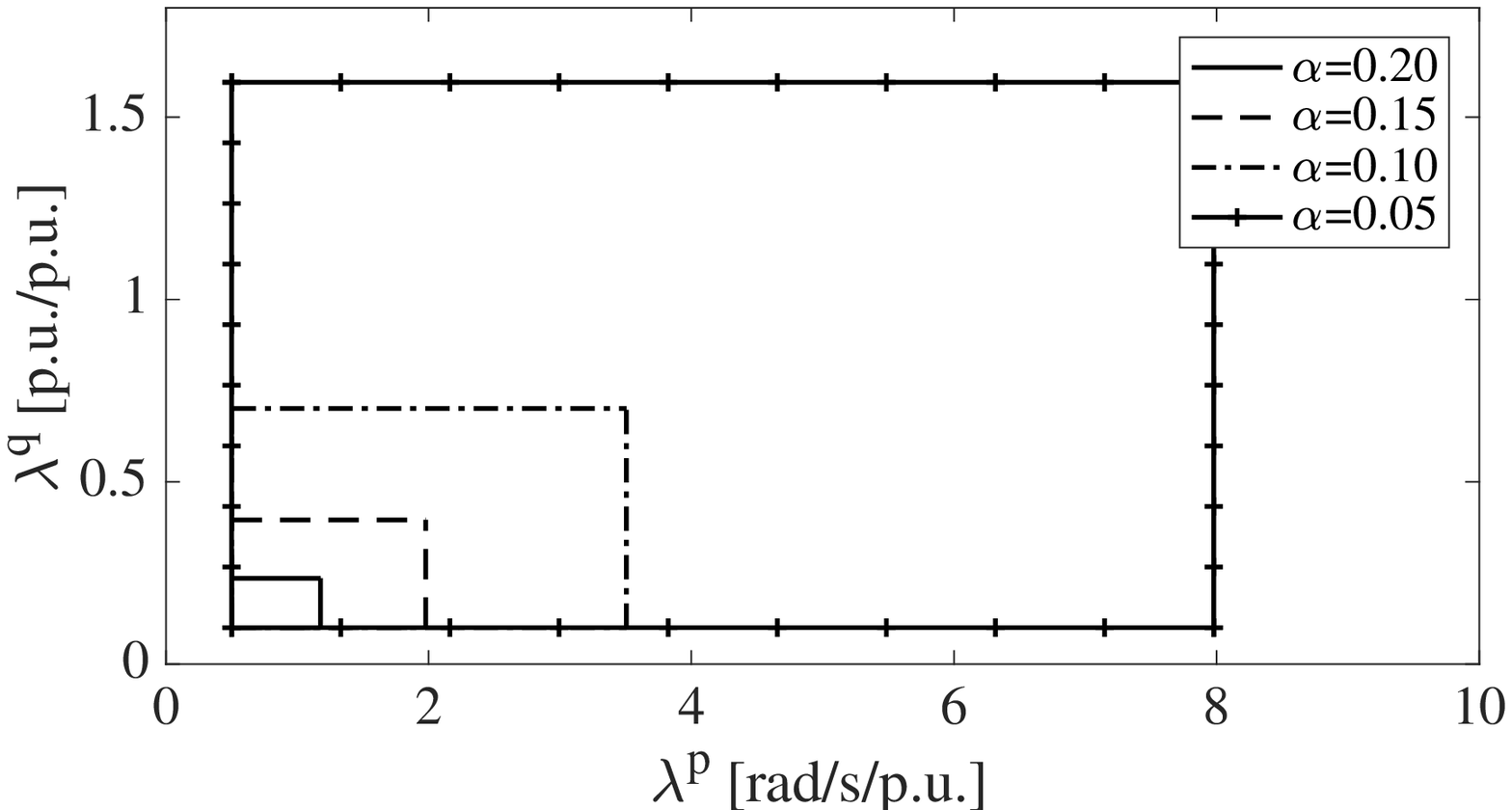}\label{F:inverter2}
}
\hspace{-0.4in}
\subfigure[parametric set for inverter \#3]{
\includegraphics[scale=0.32]{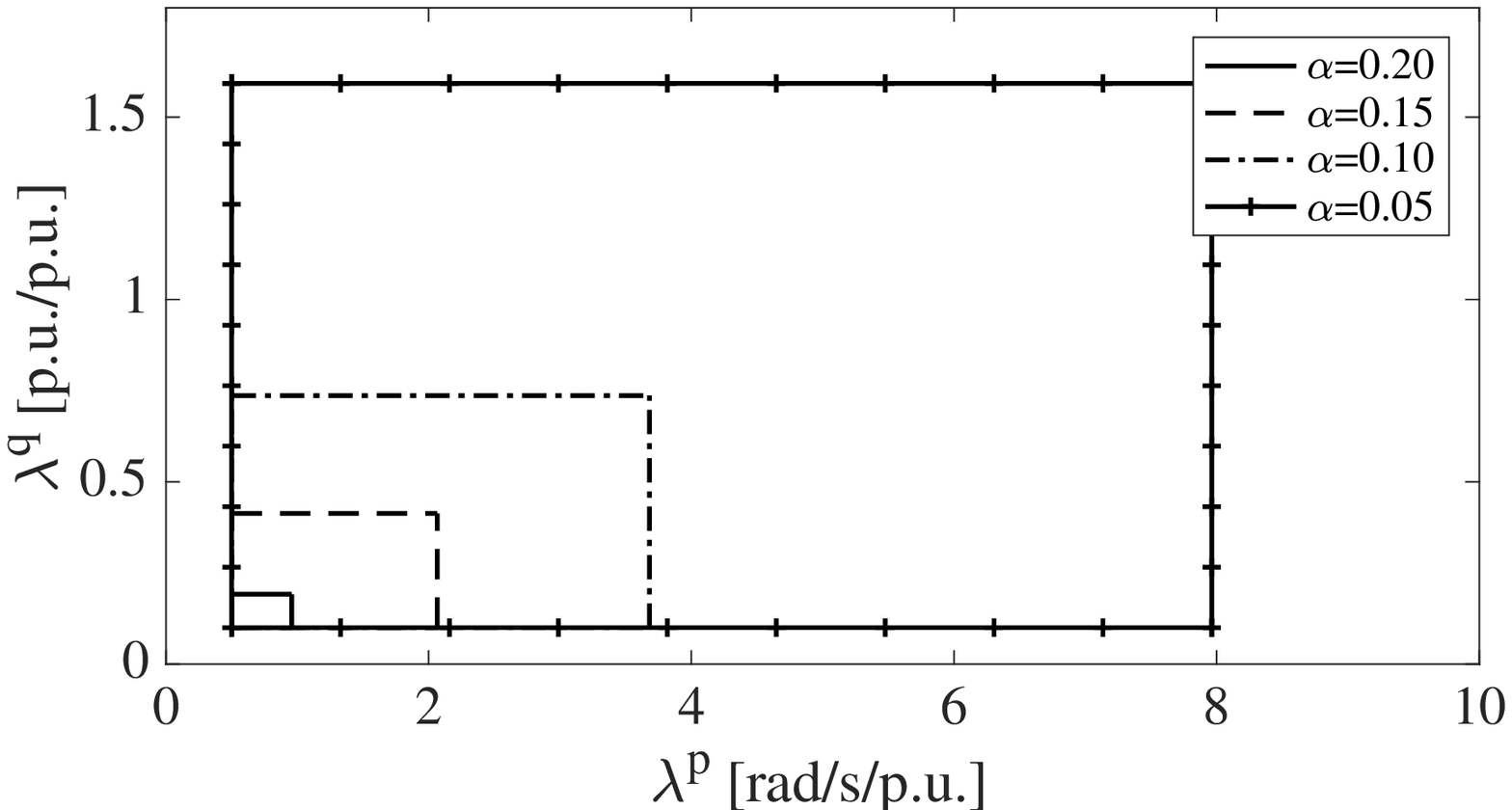}\label{F:inverter3}
}
\hspace{-0.4in}
\subfigure[parametric set for inverter \#1]{
\includegraphics[scale=0.32]{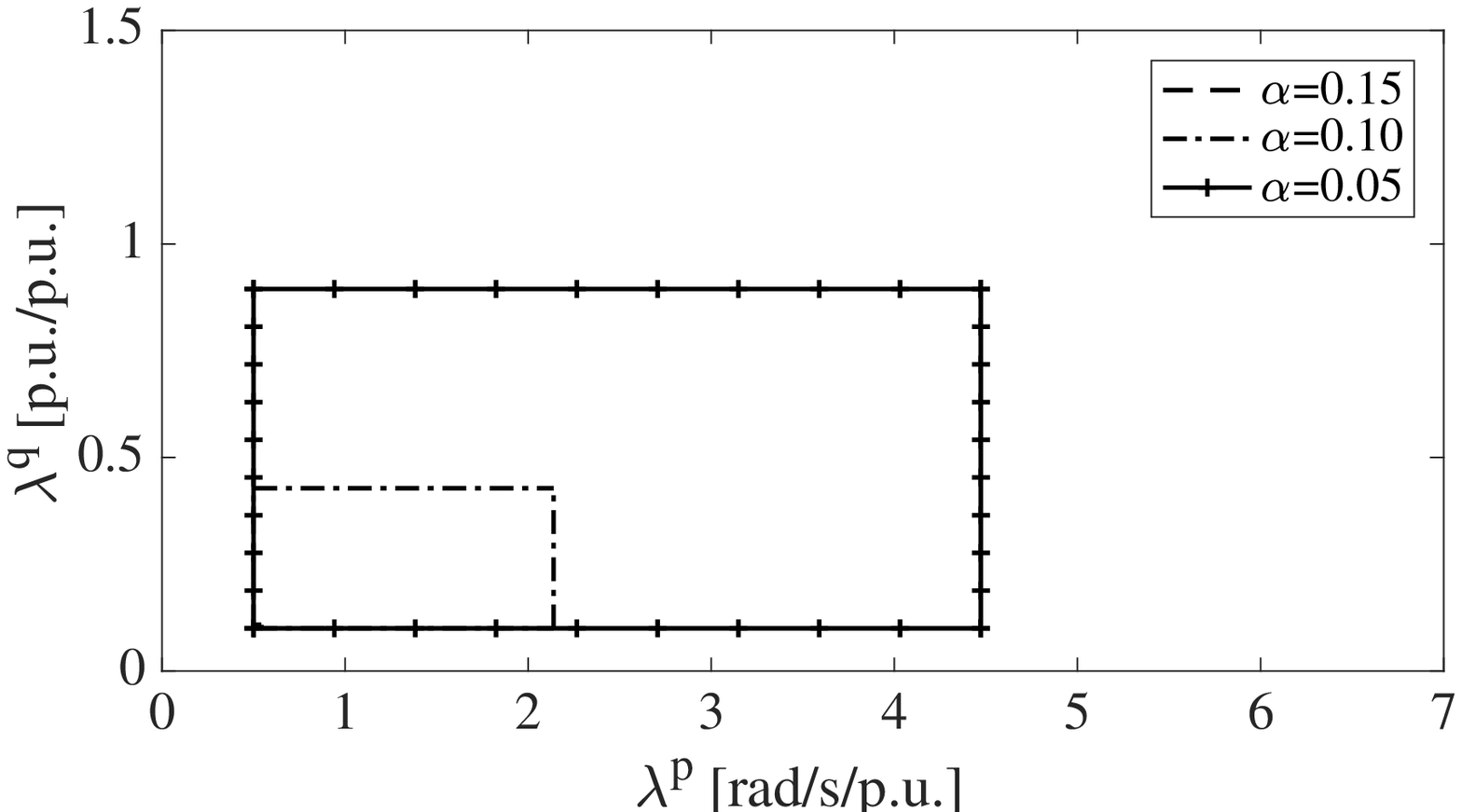}\label{F:inverter1_small}
}
\hspace{-0.4in}
\subfigure[parametric set for inverter \#2]{
\includegraphics[scale=0.32]{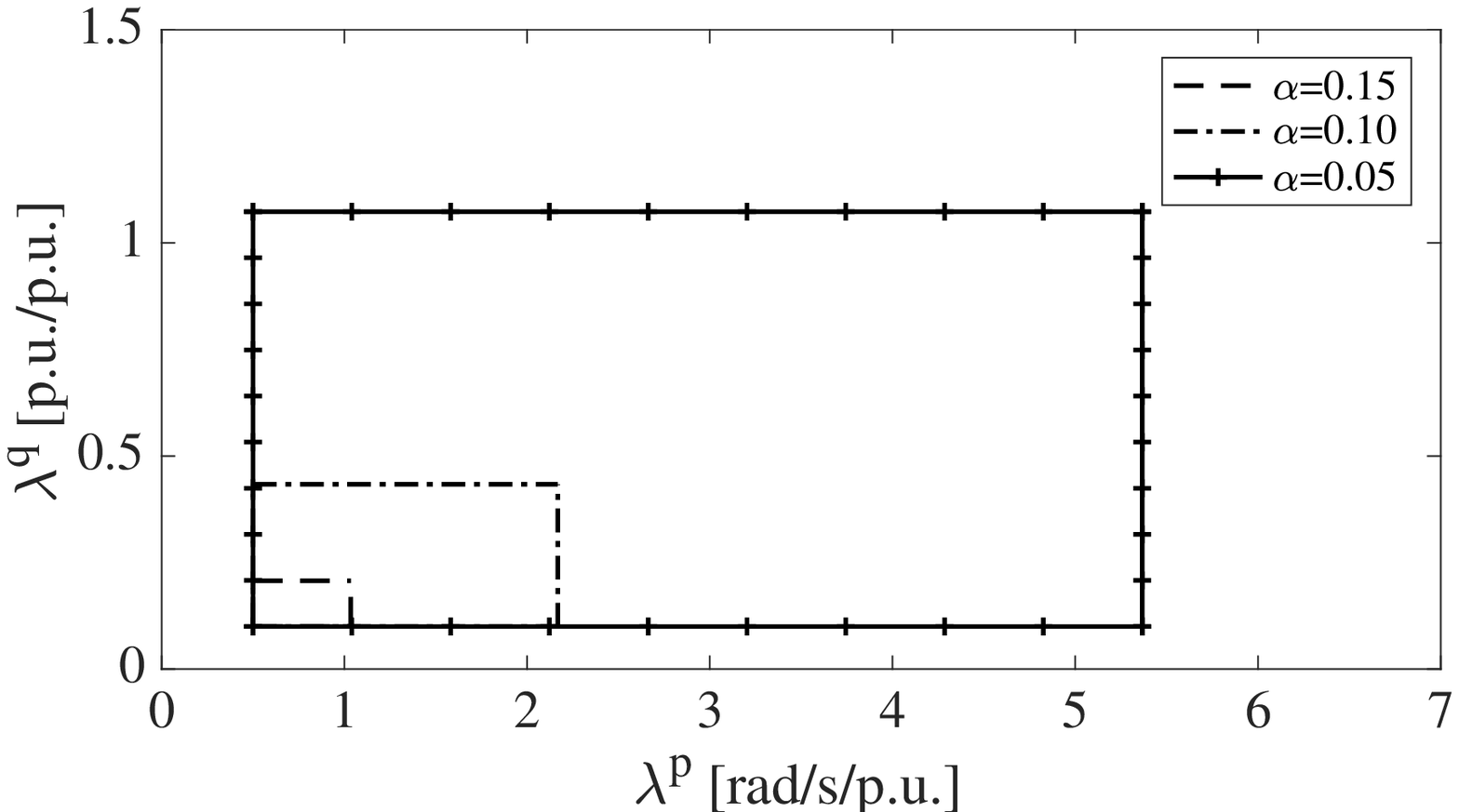}\label{F:inverter2_small}
}
\hspace{-0.4in}
\subfigure[parametric set for inverter \#3]{
\includegraphics[scale=0.32]{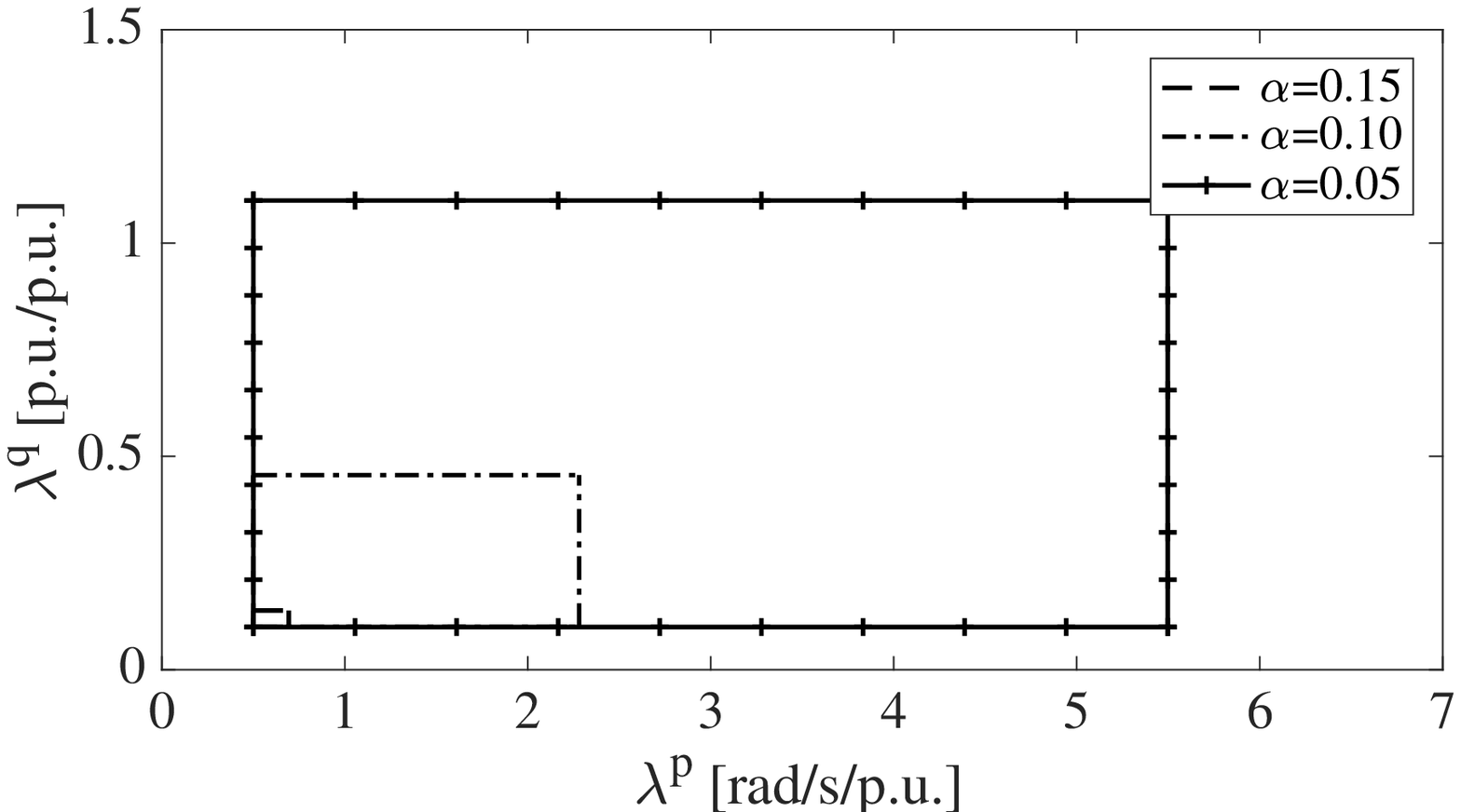}\label{F:inverter3_small}
}
\caption[Optional caption for list of figures]{Identified robust stability region for the inverter droop-coefficients under varying uncertainty levels ($\alpha$), for two different values of $c$: $c=1$  for (a)-(c), while $c=0.5$ for (d)-(e). The parametric stability region shrinks as uncertainty ($\alpha$) increases, and as allowable perturbation on the equilibrium ($c$) decreases. }
\label{F:c1}
\end{figure*}
%
%
%
%
%
%
%
%
Fig.\,\ref{F:c1} shows the identified robustly stable design parameter space for the inverters under varying uncertainties in the exogenous input, for two different values of $c$, which refer to different levels of perturbations allowed on the equilibrium point ($c=1$ allows larger perturbation than $c=0.5$)\,. The design space shrinks as the uncertainty level rises (higher value of $\alpha$) and as the allowable perturbation on the equilibrium point is reduced.

\section{Conclusion}\label{S:concl}
{In the context of robust plug-and-play {design} of nonlinear networks, we address the problem of identifying the largest region in the design parameter space that ensures asymptotic convergence of the states of the connected element under uncertainties in the network. We derive novel theoretical conditions of robust stability, as well as develop a SOS programming algorithm to identify the largest stability region in the design parameter space. Numerical illustrations are provided in the context of identifying droop-coefficient values of inverters for a plug-and-play {operation} of microgrids. Future work will explore the scalability and applicability of the algorithm to large-scale microrgid networks with other forms of dynamic resources (responsive loads, diesel generators).}


\section*{Acknowledgment}

This work was carried out under support from the U.S. Department of Energy as part of their Resilient Electric Distribution Grid R\&D program (contract DE-AC05-76RL01830).



\bibliographystyle{IEEEtran}
\bibliography{IEEEabrv,references,RefMGStability}

\begin{thebibliography}{10}
\providecommand{\url}[1]{#1}
\csname url@samestyle\endcsname
\providecommand{\newblock}{\relax}
\providecommand{\bibinfo}[2]{#2}
\providecommand{\BIBentrySTDinterwordspacing}{\spaceskip=0pt\relax}
\providecommand{\BIBentryALTinterwordstretchfactor}{4}
\providecommand{\BIBentryALTinterwordspacing}{\spaceskip=\fontdimen2\font plus
\BIBentryALTinterwordstretchfactor\fontdimen3\font minus
  \fontdimen4\font\relax}
\providecommand{\BIBforeignlanguage}[2]{{%
\expandafter\ifx\csname l@#1\endcsname\relax
\typeout{** WARNING: IEEEtran.bst: No hyphenation pattern has been}%
\typeout{** loaded for the language `#1'. Using the pattern for}%
\typeout{** the default language instead.}%
\else
\language=\csname l@#1\endcsname
\fi
#2}}
\providecommand{\BIBdecl}{\relax}
\BIBdecl

\bibitem{baheti2011cyber}
R.~Baheti and H.~Gill, ``Cyber-physical systems,'' \emph{The impact of control
  technology}, vol.~12, no.~1, pp. 161--166, 2011.

\bibitem{farhangi2010path}
H.~Farhangi, ``The path of the smart grid,'' \emph{IEEE power and energy
  magazine}, vol.~8, no.~1, 2010.

\bibitem{huang2011future}
A.~Q. Huang, M.~L. Crow, G.~T. Heydt, J.~P. Zheng, and S.~J. Dale, ``{The
  Future Renewable Electric Energy Delivery and Management (FREEDM) System: The
  Energy Internet.}'' \emph{Proceedings of the IEEE}, vol.~99, no.~1, pp.
  133--148, 2011.

\bibitem{litcofsky2012iterative}
K.~D. Litcofsky, R.~B. Afeyan, R.~J. Krom, A.~S. Khalil, and J.~J. Collins,
  ``Iterative plug-and-play methodology for constructing and modifying
  synthetic gene networks,'' \emph{Nature methods}, vol.~9, no.~11, p. 1077,
  2012.

\bibitem{bendtsen2013plug}
J.~Bendtsen, K.~Trangbaek, and J.~Stoustrup, ``Plug-and-play
  control—modifying control systems online,'' \emph{IEEE Transactions on
  Control Systems Technology}, vol.~21, no.~1, pp. 79--93, 2013.

\bibitem{nehrir2011review}
M.~Nehrir, C.~Wang, K.~Strunz, H.~Aki, R.~Ramakumar, J.~Bing, Z.~Miao, and
  Z.~Salameh, ``A review of hybrid renewable/alternative energy systems for
  electric power generation: Configurations, control, and applications,''
  \emph{IEEE Transactions on Sustainable Energy}, vol.~2, no.~4, pp. 392--403,
  2011.

\bibitem{planas2013general}
E.~Planas, A.~Gil-de Muro, J.~Andreu, I.~Kortabarria, and I.~M.
  de~Alegr{\'\i}a, ``General aspects, hierarchical controls and droop methods
  in microgrids: A review,'' \emph{Renewable and Sustainable Energy Reviews},
  vol.~17, pp. 147--159, 2013.

\bibitem{lasseter2011smart}
R.~H. Lasseter, ``Smart distribution: Coupled microgrids,'' \emph{Proceedings
  of the IEEE}, vol.~99, no.~6, pp. 1074--1082, 2011.

\bibitem{Xu:2018}
Y.~Xu, C.~Liu, K.~P. Schneider, F.~K. Tuffner, and D.~T. Ton, ``Microgrids for
  service restoration to critical load in a resilient distribution system,''
  \emph{IEEE Transactions on Smart Grid}, vol.~9, no.~1, pp. 426--437, Jan
  2018.

\bibitem{mashayekh2018security}
S.~Mashayekh, M.~Stadler, G.~Cardoso, M.~Heleno, S.~C. Madathil, H.~Nagarajan,
  R.~Bent, M.~Mueller-Stoffels, X.~Lu, and J.~Wang, ``Security-constrained
  design of isolated multi-energy microgrids,'' \emph{IEEE Transactions on
  Power Systems}, vol.~33, no.~3, pp. 2452--2462, 2018.

\bibitem{Schiffer:2014}
J.~Schiffer, R.~Ortega, A.~Astolfi, J.~Raisch, and T.~Sezi, ``Conditions for
  stability of droop-controlled inverter-based microgrids,'' \emph{Automatica},
  vol.~50, no.~10, pp. 2457--2469, 2014.

\bibitem{Vorobev:2018}
P.~Vorobev, P.~Huang, M.~A. Hosani, J.~L. Kirtley, and K.~Turitsyn, ``A
  framework for development of universal rules for microgrids stability and
  control,'' in \emph{2017 IEEE 56th Annual Conference on Decision and Control
  (CDC)}, Dec 2017, pp. 5125--5130.

\bibitem{Lyapunov:1892}
A.~M. {Lyapunov}, \emph{The General Problem of the Stability of Motion}.\hskip
  1em plus 0.5em minus 0.4em\relax Kharkov, Russia: Kharkov Math. Soc., 1892.

\bibitem{Khalil:1996}
H.~K. Khalil, \emph{{Nonlinear Systems}}.\hskip 1em plus 0.5em minus
  0.4em\relax New Jersey: Prentice Hall, 1996.

\bibitem{blanchini1999set}
F.~Blanchini, ``Set invariance in control,'' \emph{Automatica}, vol.~35,
  no.~11, pp. 1747--1767, 1999.

\bibitem{Gahinet:1996}
P.~Gahinet, P.~Apkarian, and M.~Chilali, ``Affine parameter-dependent
  {Lyapunov} functions and real parametric uncertainty,'' \emph{IEEE
  Transactions on Automatic Control}, vol.~41, no.~3, pp. 436--442, March 1996.

\bibitem{Anderson:2015}
J.~Anderson and A.~Papachristodoulou, ``Advances in computational {Lyapunov}
  analysis using sum-of-squares programming.'' \emph{Discrete \& Continuous
  Dynamical Systems-Series B}, vol.~20, no.~8, 2015.

\bibitem{Wloszek:2003}
Z.~W. Jarvis-Wloszek, ``{{Lyapunov}} based analysis and controller synthesis
  for polynomial systems using sum-of-squares optimization,'' Ph.D.
  dissertation, University of California, Berkeley, CA, 2003.

\bibitem{Parrilo:2000}
P.~A. Parrilo, ``Structured semidefinite programs and semialgebraic geometry
  methods in robustness and optimization,'' Ph.D. dissertation, Caltech,
  Pasadena, CA, 2000.

\bibitem{Tan:2006}
W.~Tan, ``Nonlinear control analysis and synthesis using sum-of-squares
  programming,'' Ph.D. dissertation, University of California, Berkeley, CA,
  2006.

\bibitem{Anghel:2013}
M.~Anghel, F.~Milano, and A.~Papachristodoulou, ``Algorithmic construction of
  {{Lyapunov}} functions for power system stability analysis,'' \emph{Circuits
  and Systems I: Regular Papers, IEEE Transactions on}, vol.~60, no.~9, pp.
  2533--2546, Sep 2013.

\bibitem{vu2017framework}
T.~L. Vu and K.~Turitsyn, ``A framework for robust assessment of power grid
  stability and resiliency,'' \emph{IEEE Transactions on Automatic Control},
  vol.~62, no.~3, pp. 1165--1177, 2017.

\bibitem{wang1998robust}
Y.~Wang, D.~J. Hill, and G.~Guo, ``Robust decentralized control for
  multimachine power systems,'' \emph{IEEE Transactions on Circuits and Systems
  I: Fundamental Theory and Applications}, vol.~45, no.~3, pp. 271--279, 1998.

\bibitem{siljak1989parameter}
D.~D. {Siljak}, ``Parameter space methods for robust control design: a guided
  tour,'' \emph{IEEE Transactions on Automatic Control}, vol.~34, no.~7, pp.
  674--688, July 1989.

\bibitem{sostools13}
A.~Papachristodoulou, J.~Anderson, G.~Valmorbida, S.~Prajna, P.~Seiler, and
  P.~A. Parrilo, ``{SOSTOOLS}: Sum of squares optimization toolbox for
  {MATLAB},'' 2013, available from
  \texttt{http://www.eng.ox.ac.uk/control/sostools}.

\bibitem{Sturm:1999}
J.~F. Sturm, ``Using {SeDuMi} 1.02, a {MATLAB} toolbox for optimization over
  symmetric cones,'' \emph{Optimization Methods and Software}, vol. 11-12, pp.
  625--653, Dec. 1999, software available at
  http://fewcal.kub.nl/sturm/software/sedumi.html.

\bibitem{Putinar:1993}
M.~Putinar, ``Positive polynomials on compact semi-algebraic sets,''
  \emph{Indiana University Mathematics Journal}, vol.~42, no.~3, pp. 969--984,
  1993.

\bibitem{Lasserre:2009}
J.-B. Lasserre, \emph{Moments, Positive Polynomials and Their
  Applications}.\hskip 1em plus 0.5em minus 0.4em\relax World Scientific, 2009,
  vol.~1.

\bibitem{Barnes:2017}
A.~Barnes, H.~Nagarajan, E.~Yamangil, R.~Bent, and S.~Backhaus, ``Tools for
  improving resilience of electric distribution systems with networked
  microgrids,'' \emph{arXiv preprint arXiv:1705.08229}, 2017.

\bibitem{Coelho:2002}
E.~A.~A. Coelho, P.~C. Cortizo, and P.~F.~D. Garcia, ``Small-signal stability
  for parallel-connected inverters in stand-alone ac supply systems,''
  \emph{IEEE Trans. on Industry Applications}, vol.~38, no.~2, pp.
  533--542, 2002.

\bibitem{Antonis:2005}
A.~Papachristodoulou and S.~Prajna, \emph{{Positive Polynomials in
  Control}}.\hskip 1em plus 0.5em minus 0.4em\relax Springer-Verlag, 2005, ch.
  Analysis of non-polynomial systems using the sum of squares decomposition,
  pp. 23--43.

\bibitem{lasseter2011certs}
R.~H. Lasseter, J.~H. Eto, B.~Schenkman, J.~Stevens, H.~Vollkommer, D.~Klapp,
  E.~Linton, H.~Hurtado, and J.~Roy, ``{CERTS microgrid laboratory test bed},''
  \emph{IEEE Transactions on Power Delivery}, vol.~26, no.~1, pp. 325--332,
  2011.

\end{thebibliography}
%
%

\end{document}